\documentclass[10pt]{article}
\usepackage{amsmath,amssymb,amsthm}
\usepackage{epsfig}
\usepackage{color}

\title{Some notes on the signed bad number in bipartite graphs} 

\author {
D.A. Mojdeh\thanks{Corresponding author} and Babak Samadi\\
Department of Mathematics\\
University of Mazandaran, Babolsar, Iran\\
{\tt damojdeh@umz.ac.ir$^*$}\\
{\tt samadibabak62@gmail.com}\vspace{3mm}\\
}
\date{}

 \newtheorem{theorem}{Theorem}[section]

\newtheorem{lemma}[theorem]{Lemma}

\theoremstyle{definition}

\theoremstyle{remark}
\newtheorem{rem}[theorem]{Remark}

\begin{document}

\maketitle
\begin{abstract}
In this paper, we deal with the signed bad number and the negative decision number of graphs. We show that two upper bounds concerning these two parameters for bipartite graphs in papers [Discrete Math. Algorithms Appl. 1 (2011), 33--41] and [Australas. J. Combin. 41 (2008), 263--272] are not true as they stand. We correct them by presenting more general bounds for triangle-free graphs by using the classic theorem of Mantel from the extremal graph theory and characterize all triangle-free graphs attaining these bounds.\vspace{.5mm}\vspace{.5mm}
\end{abstract}


\section{Introduction}

\ \ \ Throughout this paper, let $G$ be a finite graph with vertex set $V(G)$ and edge set $E(G)$. We use \cite{we} as a reference for terminology and notation which are not defined here. The {\em open neighborhood} of a vertex $v$ is denoted by $N(v)$, and the {\em closed neighborhood} of $v$ is $N[v]=N(v)\cup \{v\}$. The {\em corona} of two graphs $G_{1}$ and $G_{2}$ is the graph $G=G_{1}\circ G_{2}$ formed from one copy of $G_{1}$ and $|V(G_{1})|$ copies of $G_{2}$ where the $ith$ vertex of $G_{1}$ is adjacent to every vertex in the $ith$ copy of $G_{2}$.

Let $S \subseteq V(G)$. For a real-valued function $f:V(G)\rightarrow \mathbb{R}$ we define $f(S)=\sum_{v\in S}f(v)$. Also, $f(V(G))$ is the weight of $f$. A {\em signed bad function} ({\em bad function}), abbreviated SBF (BF), of $G$ is  a function $f:V(G)\rightarrow \{-1,1\}$ such that $f(N[v])\leq1$ ($f(N(v))\leq1$), for every $v\in V(G)$. The {\em signed bad number} ({\em negative decision number}) is $\beta_{s}(G)=\max \{f(V) | f  \mbox{\ is a SBF of}\ G \}$ ($\beta_{D}(G)=\max \{f(V) | f  \mbox{\ is a BF of}\ G \}$). Indeed, the negative decision number can be considered as the total version of the signed bad number. These two graph parameters have been studied in \cite{gks} and \cite{w}, respectively.

In 2011, Ghameshlou et al. \cite{gks} gave the following upper bound on $\beta_{s}(G)$ of a bipartite graph $G$.
 
\begin{theorem}\label{T1.1}\emph{(\cite{gks})}
If $G$ is a bipartite graph of order n, then
$$\beta_{s}(G)\leq n+2-2\lceil\sqrt{n+2}\rceil.$$
\end{theorem} 

In 2008, Wang \cite{w} exhibited the following upper bound on $\beta_{D}(G)$ of a bipartite graph $G$.

\begin{theorem}\label{T1.2}\emph{(\cite{w})}
If $G$ is a bipartite graph of order $n$, then
$$\beta_{D}(G)\leq n+3-\sqrt{4n+9}.$$
\end{theorem}

The above inequalities are not true as they stand. For example $P_{3}$ and the bistar $G=P_{2}\circ \overline{K_{3}}$ with $\beta_{s}(P_{3})=1$ and $\beta_{s}(G)=4$ are two counterexamples to Theorem \ref{T1.1} (we will show that this theorem is true for $n\neq3,8$. Indeed, Part (i) of Theorem \ref{T3.4} is a generalization and improvement of it, simultaneously). But an infinite family of counterexamples to Theorem \ref{T1.2} can be obtained as follows:\\  
For any positive integer $p$, let $G$ be a bipartite graph formed from $G'=K_{p,p}\circ \overline{K_{p+1}}$ by joining two new vertices to each pendant vertex of $G'$. Then $n=6p^{2}+8p$. It is easy to see that $f(u)=-1$ and $f(v)=1$ for all $u\in V(K_{p,p})$ and $v\in V(G)\setminus V(K_{p,p})$, respectively, defines a maximum BF of $G$ with weight $\beta_{D}(G)=6p^{2}+4p> n+3-\sqrt{4n+9}$.

In this paper, we correct the theorems by exhibiting more general results for triangle-free graphs by using the classic theorem of Mantel from the extremal graph theory. Moreover, we characterize all triangle-free graphs attaining the upper bounds.


\section{Main Theorem}

\ \ \ We need the following useful lemma.
 
\begin{lemma}\label{L1}\emph{(\cite{m}) (Mantel's Theorem)}
If $G$ is a triangle-free graph of order $n$, then 
$$|E(G)|\leq \lfloor n^{2}/4\rfloor$$
with equality if and only if $G$ is isomorphic to $K_{\lfloor\frac{n}{2}\rfloor,\lceil\frac{n}{2}\rceil}$.
\end{lemma}

Let $\Lambda$ be the family of all graphs formed from $K_{p,p}\circ \overline{K_{p+2}}$ by adding some new edges with end points in the vertices of the copies of $\overline{K_{p+2}}$ such that no triangle is induced and $\Delta(G[V(G)\setminus V(K_{p,p})])\leq1$, for some positive integer $p$.

Let $\Omega$ be the family of all graphs formed from $K_{p,p}\circ \overline{K_{p+1}}$ by adding some new edges with end points in the copies of the $\overline{K_{p+1}}$ such that no triangle is induced and $\Delta(G[V(G)\setminus V(K_{p,p})])\leq2$, for some positive integer $p$.\vspace{28mm}\\
\begin{picture}(269.518,188.518)(0,0)
\put(46,180){\circle*{6}}
\put(86,180){\circle*{6}}
\put(46,210){\circle*{6}}
\put(86,210){\circle*{6}}

\put(116,225){\circle*{6}}
\put(106,230){\circle*{6}}
\put(96,235){\circle*{6}}
\put(86,240){\circle*{6}}

\put(16,225){\circle*{6}}
\put(26,230){\circle*{6}}
\put(36,235){\circle*{6}}
\put(46,240){\circle*{6}}

\put(116,165){\circle*{6}}
\put(106,160){\circle*{6}}
\put(96,155){\circle*{6}}
\put(86,150){\circle*{6}}

\put(16,165){\circle*{6}}
\put(26,160){\circle*{6}}
\put(36,155){\circle*{6}}
\put(46,150){\circle*{6}}

\multiput(46,180)(.065,0.047){670}{\line(2,0){.9}}
\multiput(46,210)(.06,-0.045){670}{\line(2,0){.9}}
\multiput(46,180)(0,0.1){320}{\line(2,0){.9}}
\multiput(86,210)(0,-0.1){320}{\line(2,0){.9}}

\multiput(46,180)(-.044,-0.019){670}{\line(2,0){.9}}
\multiput(46,180)(-.034,-0.029){670}{\line(2,0){.9}}
\multiput(46,180)(-.017,-0.038){670}{\line(2,0){.9}}
\multiput(46,180)(-.001,-0.046){670}{\line(2,0){.9}}

\multiput(46,210)(-.029,0.027){670}{\line(2,0){.9}}
\multiput(46,210)(-.048,0.025){670}{\line(2,0){.9}}
\multiput(46,210)(-.015,0.038){670}{\line(2,0){.9}}
\multiput(46,210)(0,0.047){670}{\line(2,0){.9}}

\multiput(86,210)(0,0.047){670}{\line(2,0){.9}}
\multiput(86,210)(0.015,0.039){670}{\line(2,0){.9}}
\multiput(86,210)(0.031,0.033){670}{\line(2,0){.9}}
\multiput(86,210)(0.047,0.024){670}{\line(2,0){.9}}

\multiput(86,180)(0,-.044){670}{\line(2,0){.9}}
\multiput(86,180)(0.013,-.034){670}{\line(2,0){.9}}
\multiput(86,180)(0.032,-.032){670}{\line(2,0){.9}}
\multiput(86,180)(0.046,-.021){670}{\line(2,0){.9}}

\multiput(116,225)(-0.15,0){670}{\line(2,0){.9}}
\multiput(96,235)(-0.09,0){670}{\line(2,0){.9}}

\multiput(116,165)(-0.15,0){670}{\line(2,0){.9}}
\multiput(106,160)(-0.12,0){670}{\line(2,0){.9}}
\multiput(96,155)(-0.09,0){670}{\line(2,0){.9}}


\put(28,179){$-1$}
\put(89,178){$-1$}
\put(27,205){$-1$}
\put(88,205){$-1$}

\put(119,224){$1$}
\put(109,230){$1$}
\put(99,235){$1$}
\put(89,240){$1$}

\put(8,160){$1$}
\put(18,155){$1$}
\put(28,150){$1$}
\put(38,145){$1$}

\put(119,160){$1$}
\put(109,155){$1$}
\put(99,150){$1$}
\put(89,145){$1$}

\put(8,225){$1$}
\put(18,230){$1$}
\put(28,235){$1$}
\put(38,240){$1$}


\put(194,179){\circle*{6}}
\put(239,179){\circle*{6}}
\put(194,211){\circle*{6}}
\put(239,211){\circle*{6}}
\put(284,179){\circle*{6}}
\put(284,211){\circle*{6}}

\put(308,225){\circle*{6}}
\put(300,232){\circle*{6}}
\put(292,238){\circle*{6}}
\put(284,244){\circle*{6}}

\put(257,244){\circle*{6}}
\put(246,244){\circle*{6}}
\put(234,244){\circle*{6}}
\put(223,244){\circle*{6}}

\put(257,147){\circle*{6}}
\put(246,147){\circle*{6}}
\put(234,147){\circle*{6}}
\put(223,147){\circle*{6}}

\put(308,159){\circle*{6}}
\put(300,154){\circle*{6}}
\put(292,150){\circle*{6}}
\put(284,146){\circle*{6}}

\put(170,224){\circle*{6}}
\put(176,231){\circle*{6}}
\put(183,237){\circle*{6}}
\put(191,243){\circle*{6}}

\put(191,148){\circle*{6}}
\put(183,152){\circle*{6}}
\put(175,156){\circle*{6}}
\put(167,160){\circle*{6}}

\multiput(239,211)(.024,.048){670}{\line(2,0){.9}}
\multiput(239,211)(-.024,.047){670}{\line(2,0){.9}}

\multiput(194,179)(-.004,-.049){670}{\line(2,0){.9}}
\multiput(194,179)(-.04,-.026){670}{\line(2,0){.9}}

\multiput(239,179)(-.025,-.049){670}{\line(2,0){.9}}
\multiput(239,179)(.025,-.049){670}{\line(2,0){.9}}

\multiput(284,211)(.037,.022){670}{\line(2,0){.9}}
\multiput(284,211)(-.002,.052){670}{\line(2,0){.9}}

\multiput(194,211)(-.005,.052){670}{\line(2,0){.9}}
\multiput(194,211)(-.04,.021){670}{\line(2,0){.9}}

\multiput(283,179)(-.002,-.052){670}{\line(2,0){.9}}
\multiput(283,179)(.04,-.032){670}{\line(2,0){.9}}

\multiput(283,179)(0,.052){670}{\line(2,0){.9}}
\multiput(283,179)(-.069,.048){670}{\line(2,0){.9}}
\multiput(283,179)(-.14,.05){670}{\line(2,0){.9}}

\multiput(194,179)(0,.052){670}{\line(2,0){.9}}
\multiput(194,179)(.07,.052){670}{\line(2,0){.9}}
\multiput(194,179)(.142,.052){670}{\line(2,0){.9}}

\multiput(239,179)(-.076,.052){670}{\line(2,0){.9}}
\multiput(239,179)(-.001,.052){670}{\line(2,0){.9}}
\multiput(239,179)(.074,.052){670}{\line(2,0){.9}}

\multiput(194,179)(.058,-.045){670}{\line(2,0){.9}}
\multiput(194,179)(-.026,.077){670}{\line(2,0){.9}}

\multiput(239,179)(.081,-.043){670}{\line(2,0){.9}}
\multiput(239,179)(-.088,-.039){670}{\line(2,0){.9}}

\multiput(284,179)(-.056,-.046){670}{\line(2,0){.9}}
\multiput(284,179)(.023,.08){670}{\line(2,0){.9}}

\multiput(239,211)(-.088,.04){670}{\line(2,0){.9}}
\multiput(239,211)(.079,.04){670}{\line(2,0){.9}}

\multiput(284,211)(.023,-.088){670}{\line(2,0){.9}}
\multiput(284,211)(-.059,.051){670}{\line(2,0){.9}}

\multiput(194,211)(.06,.052){670}{\line(2,0){.9}}
\multiput(194,211)(-.03,-.084){670}{\line(2,0){.9}}

\multiput(257,244)(-.05,0){670}{\line(2,0){.9}}
\multiput(257,147)(-.05,0){670}{\line(2,0){.9}}

\multiput(308,225)(-.015,.012){670}{\line(2,0){.9}}
\multiput(292,238)(-.015,.012){670}{\line(2,0){.9}}

\multiput(308,159)(-.015,-.009){670}{\line(2,0){.9}}
\multiput(292,150)(-.015,-.009){670}{\line(2,0){.9}}

\multiput(168,224)(.014,.014){670}{\line(2,0){.9}}
\multiput(183,237)(.014,.014){670}{\line(2,0){.9}}

\multiput(191,148)(-.014,.008){670}{\line(2,0){.9}}
\multiput(175,156)(-.014,.008){670}{\line(2,0){.9}}

\put(199,175){$-1$}
\put(244,176){$-1$}
\put(199,209){$-1$}
\put(242,208){$-1$}
\put(264,174){$-1$}
\put(265,210){$-1$}

\put(309,227){$1$}
\put(301,234){$1$}
\put(293,240){$1$}
\put(285,246){$1$}

\put(257,247){$1$}
\put(246,247){$1$}
\put(234,247){$1$}
\put(223,247){$1$}

\put(259,138){$1$}
\put(247,138){$1$}
\put(235,138){$1$}
\put(224,138){$1$}

\put(308,150){$1$}
\put(300,145){$1$}
\put(292,141){$1$}
\put(284,137){$1$}

\put(165,228){$1$}
\put(170,235){$1$}
\put(180,241){$1$}
\put(187,247){$1$}

\put(191,136){$1$}
\put(183,142){$1$}
\put(175,146){$1$}
\put(167,150){$1$}

\end{picture}\vspace{-53mm}\\
\begin{center}
A member of $\Lambda$ for $p=2$\ \ \ \ \ \ \ \ \ \ \ \ \ A member of $\Omega$ for $p=3$
\end{center}\vspace{1.3mm}
\ \ \ For convenience, we make use of the following notation. Let $G$ be a graph and $f:V(G)\longrightarrow\{-1,1\}$ be a SBF or BF of $G$. Define $V_{+}=\{v\in V(G) \mid f(v)=1 \}$ and $V_{-}=\{v\in V(G) \mid f(v)=-1 \}$. Let $[V_{+},V_{-}]$ be the set of edges having one end point in $V_{+}$ and the other in $V_{-}$.

We are now in a position to present the main theorem of the paper.

\begin{theorem}\label{T3.4}
If $G$ is a triangle-free graph of order $n$, then\vspace{1mm}\\
\emph{(i)}\ If $\delta(G)\geq1$, then $\beta_{s}(G)\leq n+6-2\sqrt{9+2n}$.\vspace{1mm}\\
\emph{(ii)}\ If $\delta(G)\geq2$, then $\beta_{D}(G)\leq n+4-2\sqrt{4+2n}$.\vspace{1mm}\\
Furthermore, the first inequality holds with equality if and only if $G\in \Lambda$ and the second one holds with equality if and only if $G\in \Omega$.
\end{theorem}
\begin{proof}
(i)\ Let $f$ be a maximum SBF of $G$. It follows that, every vertex in $V_{+}$ has at least one neighbor in $V_{-}$. Also, $|N(v)\cap V_{+}|\leq|N(v)\cap V_{-}|+2$ for all $v\in V_{-}$. Furthermore, by Lemma \ref{L1} we have
\begin{equation}\label{EQ11}
\begin{array}{lcl}
|V_{+}|&\leq& |[V_{-},V_{+}]|=\sum_{v\in V_{-}}|N(v)\cap V_{+}|\leq \sum_{v\in V_{-}}(|N(v)\cap V_{-}|+2)\\
&=&2|E(G[V_{-}])|+2|V_{-}|\leq|V_{-}|^{2}/2+2|V_{-}|.
\end{array}
\end{equation}
Therefore, 
$$|V_{-}|^2+6|V_{-}|-2n\geq0.$$
Solving the above inequality for $|V_{-}|$ we obtain
$$(n-\beta_{s}(G))/2=|V_{-}|\geq-3+\sqrt{9+2n},$$
implying the desired upper bound.

Let $G\in \Lambda$. We define $f:V(G)\rightarrow\{-1,1\}$ by, 
$$f(v)=\left \{
\begin{array}{lll}
-1 & \mbox{;} & v\in V(K_{p,p}) \\
 \ 1 & \mbox{;} & \mbox{otherwise}.
\end{array}
\right.$$
It is easy to check that $f$ is a SBF of $G$ with weight $f(V(G))=\beta_{s}(G)=n+6-2\sqrt{9+2n}$.

Now let $G$ be a graph for which the equality holds and $f$ be a maximum SBF of $G$. By the inequality (\ref{EQ11}), every vertex in $V_{+}$ must have exactly one neighbor in $V_{-}$. Also, $2|E(G[V_{-}])|=|V_{-}|^{2}/2$ along with Lemma \ref{L1} implies that $G[V_{-}]=K_{|V_{-}|/2,|V_{-}|/2}$. Moreover, $|N(v)\cap V_{+}|=|N(v)\cap V_{-}|+2$ for all $v\in V_{-}$ implies that every vertex in $V_{-}$ is adjacent to exactly $|V_{-}|/2+2$ vertices in $V_{+}$. Suppose to the contrary that, there exists a vertex $v$ of the subgraph $H$ induced by $V(G)\setminus V(K_{|V_{-}|/2,|V_{-}|/2})$ with deg$_{H}(v)\geq2$. Then $f(N[v])\geq2$, a contradiction.\\
(ii)\ The proof is almost along the lines of (i). Let $f$ be a maximum BF of $G$. Since $\delta(G)
\geq2$, every vertex in $V_{+}$ has at least one neighbor in $V_{-}$. Also, $|N(v)\cap V_{+}|\leq|N(v)
\cap V_{-}|+1$ for all $v\in V(G)$. Similar to the inequality chain (\ref{EQ11}) we conclude that $$|V_{-}|^2+4|V_{-}|-2n\geq0,$$
which implies the desired upper bound.

If $G\in \Omega$, then we can obtain a BF of $G$ with weight $\beta_{D}(G)=n+4-2\sqrt{4+2n}$ by assigning $-1$ to the vertices in $K_{p,p}$ and $1$ to the other ones.
 
We now let the graph $G$ attain the upper bound by the BF $f$ of it. Similar to the part (i), we have $G[V_{-}]=K_{|V_{-}|/2,|V_{-}|/2}$. Also, every vertex in $V_{+}$ has exactly one neighbor in $V_{-}$ and $|N(v)\cap V_{+}|=|N(v)\cap V_{-}|+1$ for all $v\in V_{-}$. Finally, since $f(N(v))\leq1$, $|N(v)\cap(V(G)\setminus V(K_{|V_{-}|/2,|V_{-}|/2}))|\leq2$ for each $v\in V(G)\setminus V(K_{|V_{-}|/2,|V_{-}|/2})$. Therefore, $G\in \Omega$. This completes the proof. 
\end{proof}
\begin{rem}
The proof of Theorem \ref{T1.1} in \cite{gks} contains a gap. For the sake of completeness, we point it out. As it is presented in \cite{gks}:\\
"Let $f$ be a maximum SBF of the bipartite graph $G$ with bipartition $X_{1}$ and $X_{2}$. Define $X^-_{i}=X_{i}\cap V_{-}$ and $X^+_{i}=X_{i}\cap V_{+}$ for $i=1,2$." They claimed that $|X^+_{1}|\leq|X^-_{2}|+|X^-_{2}||X^-_{1}|$ and $|X^+_{2}|\leq|X^-_{1}|+|X^-_{2}||X^-_{1}|$. These two inequalities do not hold in general, as non of them are true for the bistar $P_{2}\circ \overline{K_{3}}$ and one of them is not true for $P_{3}$. 
Considering Theorem \ref{T3.4}, we have
\begin{equation*}
\beta_{s}(G)\leq n+6-\lceil2\sqrt{9+2n}\rceil\leq n+2-2\lceil\sqrt{n+2}\rceil
\end{equation*}
holds for each integer $1<n\notin\{3,4,5,8,9,10,15,16\}$. So, Theorem \ref{T1.1} is true for all bipartite graphs of the identified orders. On the other hand, $\beta_{s}(G)$ and $n$ have the same parity. Therefore, Part (i) of Theorem \ref{T3.4} implies that $\beta_{s}(G)\leq 0,1,3,4,7,8$ for $n=4,5,9,10,15,16$, respectively. This coincides with the upper bounds on $\beta_{s}(G)$, of a bipartite graph $G$ of these orders, by Theorem \ref{T1.1}.
\end{rem}


\end{document}